\documentclass[11pt]{amsart}
\usepackage{mathrsfs}
\usepackage{amssymb}
\usepackage{graphicx}
\usepackage[T1]{fontenc}
\usepackage[all]{xy}
\usepackage{amscd}
\usepackage{amsmath, amssymb}
\usepackage{amsthm}
\usepackage{amsfonts}
\usepackage[colorlinks,linkcolor=blue,citecolor=blue, pdfstartview=FitH]
{hyperref}
\usepackage{amscd}
\usepackage{easybmat}
\usepackage{mathrsfs}
\usepackage{amsfonts}
\usepackage{color}
\usepackage{pifont}
\usepackage{upgreek}
\usepackage{bm}
\usepackage{hyperref}
\usepackage{shorttoc}
\usepackage{amsmath,amstext,amsthm,a4,amssymb,amscd}
\usepackage[mathscr]{eucal}
\usepackage{mathrsfs}
\usepackage{epsf}
\usepackage{tikz}

\numberwithin{equation}{section}

\newcommand{\R}{\ensuremath{\mathbb{R}}}

\def\p{\partial}
\def\o{\overline}
\def\b{\bar}
\def\mb{\mathbb}
\def\mc{\mathcal}
\def\mr{\mathrm}
\def\n{\nabla}

\def\cal{\mc}
\def\ra{\rightarrow}

\newtheorem{thm}{Theorem}[section]
\newtheorem{lemma}[thm]{Lemma}
\newtheorem{prop}[thm]{Proposition}
\newtheorem{cor}[thm]{Corollary}
\theoremstyle{definition}
\newtheorem{rem}[thm]{Remark}

\theoremstyle{definition}
\newtheorem{defn}[thm]{Definition}

\theoremstyle{plain}
\newcommand{\comment}[1]{}

\usepackage{amscd}
\usepackage{fancyhdr}
\newenvironment{aligns}{\equation\aligned}{\endaligned\endequation}
\makeatletter
\@namedef{subjclassname@2020}{2020 Mathematics Subject Classification}

\begin{document}

\title{Curvature of the total space of a Griffiths negative vector bundle and quasi-Fuchsian space}

\author{InKang Kim}
\author{Xueyuan Wan}
\author{Genkai Zhang}

\address{Inkang Kim: School of Mathematics, KIAS, Heogiro 85, Dongdaemun-gu Seoul, 02455, Republic of Korea}
\email{inkang@kias.re.kr}

\address{Xueyuan Wan: Mathematical Science Research Center, Chongqing University of Technology, Chongqing 400054, China}
\email{xwan@cqut.edu.cn}

\address{Genkai Zhang: Mathematical Sciences, Chalmers University of Technology and Mathematical Sciences, G\"oteborg University, SE-41296 G\"oteborg, Sweden}
\email{genkai@chalmers.se}

\begin{abstract}
  For a holomorphic
  vector bundle $E$ over
  a Hermitian  manifold $M$ there
  are two important notions 
  of curvature  positivity, the Griffiths positivity
  and Nakano positivity. We study the consequence
  of these positivities and the relevant estimates.
  If $E$ is Griffiths negative over  K\"ahler manifold,
  then there is a K\"ahler metric on its total space $E$, and
 we calculate the curvature and prove the non-positivity of the
 curvature along the tautological direction. The Nakano positivity
 can be formulated as a positivity for the Nakano curvature
 operator and we give estimate the Nakano curvature operator associated with a Nakano positive direct image bundle. As applications we construct a  mapping class group invariant K\"ahler metric on the quasi-Fuchsian space $\mr{QF}(S)$, which extends the Weil-Petersson metric on
 the Teichm\"uller space $\mathcal T(S)\subset \mr{QF}(S)$,
 and  we obtain  estimates for the Nakano curvature
  operator for the dual Weil-Petersson metric on the holomorphic cotangent bundle of Teichm\"uller space. 

 \end{abstract}

 \subjclass[2020]{32G15, 32Q15, 34L15}  
 \keywords{Griffiths negativity, Teichm\"uller space, quasi-Fuchsian space, complex projective structure, Nakano curvature operator, Weil-Petersson metric}
 \thanks{Research by Inkang Kim was partially supported by Grant
   NRF-2019R1A2C1083865 and KIAS Individual Grant (MG031408),
    Xueyuan Wan by  NSFC (No. 12101093), Scientific Research
    Foundation of Chongqing University of Technology,
    and Genkai Zhang by  Swedish Research Council (VR), 2018-03402.
    Part of this work was done when  G. Zhang was
    visiting KIAS in August 2022 as a KIAS scholar and the support
    from KIAS is greatly acknowledged.}

\maketitle
\tableofcontents

\section*{Introduction}

Let $E$ be a holomorphic Hermitian vector bundle
over a K\"ah{}ler manifold  $M$.
There are two important notions 
  of curvature  positivity, the Griffiths positivity
  and Nakano positivity. The aim of the present
  paper is to study the consequences
  of these positivities and apply the results
  to various cases related to the
  Teichm\"u{}ller space.

  The curvature on a holomorphic bundle  $E$ over
  $M$ can be viewed as an operator on $  T^{(1, 0)}_mM \otimes E_m $,
  $m\in M$,  and it is self-adjoint.
  It is called  Griffiths positive
  if $R$
  is positive on simple tensors $u\otimes v\in
  T^{(1, 0)}_m\otimes E_m$,   and 
  Nakano positive if it is positive on the total tensor space
  $T^{(1, 0)}_m\otimes E_m$, $m\in M$.
  We prove first that if $E$ is Griffith
  negative   then there is a K\"ahler metric on its total space $E$, and
 we calculate the curvature and prove the non-positivity of the
 curvature along the fiber direction. We
 then give estimates of the  Nakano curvature operator
 associated with a Nakano positive direct image bundle.

Let $p:\mc{X}\to M$ be a holomorphic fibration with compact fibers. Suppose the relative canonical line bundle $K_{\mc{X}/M}$ is  positive over $\mc{X}$, and consider the following direct image bundle 
\begin{align*}
\begin{split}
  E=p_*(K_{\mc{X}/M}^{\otimes 2}). 
 \end{split}
\end{align*}
Following Berndtsson \cite{Bo, Bo1}, there exists a natural $L^2$-metric (a Hermitian metric) on the holomorphic vector bundle $E$, and the curvature of the $L^2$-metric is Nakano positive. It is natural to consider the extension of Nakano's positivity. This suggests we estimate the Nakano curvature operator $Q$, which is defined as the quadratic form on $TM\otimes E$
 $$Q(A,B)=\sum_{\alpha,j}a^{\alpha i}\o{b^{\beta j}}\left\langle R(\frac{\p}{\p z^\alpha},\frac{\p}{\p \b{z}^\beta})e_i,e_j\right\rangle$$
  for any $A=a^{\alpha i} \frac{\p}{\p z^\alpha}\otimes e_i$ and $B=b^{\beta j}\frac{\p}{\p z^\beta}\otimes e_j$ in $TM\otimes E$.  We identify it  with an operator on $TM\otimes E$
   denoted also by $Q$, 
  $$
\langle Q(A), B\rangle_{TM\otimes E}=Q(A,B). 
  $$
  Let $\iota: T_tM\otimes E_t \to A^{n-1, 1}(X_t)$, $n=\dim X_t$, 
be the diagonal map 
\begin{equation}\label{iota}
\iota: 
T_tM\otimes E_t \to 
A^{0,1}(X_t,TX_t)\otimes H^{0}(X_t,L_t\otimes  K_t) 
\to A^{n-1, 1}(X_t, L_t), 
\end{equation}
where the first map is given by the  Kodaira-Spencer tensor 
and the second one is the evaluation, see \eqref{KS tensor}. 

We obtain the following estimates on the Nakano curvature operator. 
\begin{thm}\label{thm0.2}
For any fixed $t\in M$,  let 
 $\sigma$ be the maximum of the eigenvalue 
 of $\Box'=\n'\n'^*+\n'^*\n$ on the finite-dimensional subspace 
 $\iota( T_tM\otimes E_t )$. 
 Then we have 
  $$
Q(A, A)\ge   (\frac 1{1+2n}  + (1+\sigma)^{-1}) 
  \Vert \iota(A)\Vert^2. 
  $$
  In particular, if the map $\iota$ is injective, then the Nakano curvature operator 
$Q$ satisfies 
  $$
Q\ge   (\frac 1{1+2n}  + (1+\sigma)^{-1})  \lambda_{min} 
$$
where 
$\lambda_{min}$  is the lowest eigenvalue of $\iota^\ast \iota$
with respect to the Hilbert space norms 
in $T_tM\otimes E_t$ and $A^{n-1, 1}(X_t,L_t)$. 
 \end{thm}

 We consider then the more concrete  case
 related to these notions.

 One of the most studied cases
 of the above notions
 is the case when $M$ is the
 Teichm\"uller space $\cal T(S)$
 equipped with the Weil-Petterson metric of a surface $S$
  and $E$ the tangent bundle
 or relevant bundles. There have been active studies on the properties of this metric since its birth.
More recently, some 
new K\"ahler metrics with more desirable properties
such as K\"ahler hyperbolicity have been found where the
K\"ahler hyperbolicity means that the K\"ahler metric is complete with  bounded curvatures
and it has a bounded  K\"ahler primitive.
Such  K\"ahler hyperbolic metrics are studied by
McMullen \cite{Mc} and Liu-Sun-Yau \cite{Yau}.

In Kleinian group theory, the quasi-Fuchsian space $\mr{QF}(S)$ is a
quasi-conformal deformation space of the Fuchsian space $\mr{F}(S)$
which can be identified with $\cal T(S)$. By Bers' simultaneous
uniformization theorem, $\mr{QF}(S)$ can be naturally identified with
$\cal T(S)\times \cal T(\bar S)$ where $\bar S$ is $S$ with reversed
orientation. With this identification, the mapping class group acts diagonally on $\mr{QF}(S)$ and $\mr{F}(S)=\cal T(S)$ sits diagonally on $\cal T(S)\times \cal T(\bar S)$. But this diagonal embedding is totally real. Hence if one gives a product K\"ahler metric on $\mr{QF}(S)$, this metric is not an extension of a K\"ahler metric on $\mr{F}(S)$.
There have been several attempts to extend a K\"ahler metric of $\cal T(S)$ to $\mr{QF}(S)$.
Bridgeman and Taylor \cite{BT} described a quasi-metric which extends the K\"ahler metric of $\cal T(S)$, but
it vanishes along the pure bending deformation vectors \cite{Br}.

As an application of our result we give a completely new mapping class group invariant K\"ahler metric on $\mr{QF}(S)$ which extends any K\"ahler metric on $\cal T(S)$. Indeed such a metric is already defined in the paper \cite{KZ} a few years ago.
The metric is defined by a K\"ahler potential, which is a combination of $L^2$-norm of a fiber and a K\"ahler potential on the base $\cal T(S)$. We will see that $\mr{QF}(S)$ can be embedded, via Bers embedding using the complex projective structures, in the holomorphic  cotangent bundle $\cal Q(S)$ over $\cal T(S)$
with fibers being quadratic holomorphic differentials, as a bounded open neighborhood of the zero section.

 On the holomorphic vector bundle
$\cal Q(S)$, one can give
a mapping class group invariant K\"ahler metric as
follows. By a theorem of Berndtsson \cite{Bo}, one can show that $\cal
Q(S)$ is Griffiths positive with respect to the $L^2$-metric; see \cite{KZ} for a proof.
Hence its dual bundle $\cal B(S)=\cal Q^*(S)$, which is a tangent
bundle of Teichm\"uller space whose fiber is the set of Beltrami
differentials, is Griffiths negative. We fix in the rest of the paper
this realization of $\mr{QF}(S)$ as
a subset in $\cal B(S)$.
We recall that the $L^2$-norm of a Beltrami differential $w=\mu(v) \frac{ d\bar v}{d v}$ is given by
$$||w||^2=(w,w)=\int_Y |\mu(v)|^2 \rho(v)^2|dv|^2$$ where $v$ is a local holomorphic coordinate on $Y$ and $\rho(v)|dv|$ is a hyperbolic metric on $Y$.  Here $(,)$ denotes the $L^2$ inner product over each fiber and $||\cdot||$ denotes its associated norm.

  The K\"ahler metric depending on a constant $k>0$ and a K\"ahler metric on $\cal T(S)$, is constructed on $\cal B(S)$ via K\"ahler potential
$$\Phi(w)=||w||^2 + k \pi^* \psi(w),$$ where $w$ is an element in the fiber, $\psi$ is a K\"ahler potential on $\cal T(S)$ and $\pi:\cal B(S)\ra \cal T(S)$ is a projection.

In  local holomorphic coordinates $(z,x)$ around $w_0$, where $z=(z_1,\cdots,z_{3g-3})$ is  local holomorphic coordinates around $\pi(w_0)=z_0$, and $w=\sum x^\alpha e_\alpha(z)$ with respect to  local holomorphic sections $e_\alpha$, for a holomorphic tangent vector at $w$ $T=u+v$ with a canonical decomposition into $\cal T(S)$ direction $u$ and vertical fiber direction $v$, the norm of $T$ with respect to the K\"ahler metric  defined by the K\"ahler potential $\Phi$ is given by
$$||T||_\Phi^2=\bar\partial_T \partial_T \Phi(w)= - (R(u,\bar u)w,w)+ (\cal D_u w+ v, \cal D_u w+v)+k \bar\partial_u
 \partial_u \psi >0,$$ where $R$ is a curvature of the Chern
 connection $\nabla$ on $\cal B(S)$ and $\nabla=\cal D +\bar\partial$
 is a decomposition into $(1,0)$ and $(0,1)$ part of the connection.
 See \cite{KZ} for details.
 
The K\"ahler metric we construct on $\mr{QF}(S)$ is the restriction to this open neighborhood.  When we choose  the Weil-Petersson metric on $\cal T(S)$, we
show that the new K\"ahler metric on $\cal B(S)$ has similar properties, such as its K\"ahler form has a bounded primitive 
and the curvature has non-positivity in some directions.

\begin{thm}\label{main theorem} There exists a mapping class group invariant  K\"ahler metric on $\cal B(S)$,  which extends the Weil-Petersson metric on $\cal T(S)\subset \cal B(S)$. Consequently, the quasi-Fuchsian space $\mr{QF}(S)$ has such a K\"ahler metric as an invariant open set of $\cal B(S)$ under the mapping class group. Furthermore, the curvature of the metric is non-positive when evaluated on the tautological section (and vanishes along vertical directions); its Ricci curvature is bounded from above by $-\frac{1}{\pi(g-1)}$ when restricted to Teichm\"uller space, and its K\"ahler form has a bounded primitive.
\end{thm}

We specify Theorem \ref{thm0.2} to  the case of
Teichm\"u{}ller space $M=\mc{T}$. Let $\mc{X}$ be Teichm\"uller curve
over Teichm\"uller space
and consider the bundle $$E=p_*(K_{\mc{X}/\mc{T}}^{\otimes 2}),$$
which is exactly the cotangent bundle $T^*\mc{T}$ of $\mc{T}$. In this case, the map $\iota$ is an isometric embedding and $n=1$. Hence 
\begin{cor}\label{cor0.3}
We have the following lower estimate
  for the Nakano curvature operator,
  $$  Q(A, A)\ge (\frac 13  + \frac 1{1+\sigma})
  \Vert A\Vert^2.$$
    As an operator on the tensor product $T^{(1, 0)}_m\otimes E_m$, we have
  $ Q\ge \frac 13  + \frac 1{1+\sigma}$.  In particular, if $A$ has the form $A=\sum_i\lambda_i \mu_i\otimes q_i$ with $\lambda_i\geq 0$, then 
  $$Q(A,A)\geq \frac{2}{3}\|A\|^2,$$
  where $q_i$ is a holomorphic quadratic differential and $\mu_i$ denotes the associated harmonic Beltrami differential of $q_i$.
\end{cor} 

This article is organized as follows. In Section \ref{sec1}, we will review the horizontal and vertical decomposition associated with a Hermitian vector bundle, the definition of a Griffiths negative vector bundle, and the K\"ahler metric on the total space. In Section \ref{sec2}, we will calculate the curvature of the K\"ahler metric and obtain some curvature properties. In Section \ref{sec5}, we will give some estimates on the Nakano curvature operator associated with a direct image bundle and prove Theorem \ref{thm0.2} and Corollary \ref{cor0.3}. In Section \ref{sec3},  we will recall the definitions of quasi-Fuchsian space and complex projective structure, and we will embed the quasi-Fuchsian space into the space of complex projective structures. Then we will define a mapping class group invariant K\"ahler metric on the quasi-Fuchsian space and prove Theorem \ref{main theorem}. In Section \ref{sec4}, we discuss some K\"ahler metrics on other geometric structures. In the appendix,  we will derive the curvature formula of the Weil-Petersson metric on Teichm\"uller space by using Berndtsson's method \cite[Section 4.2]{Bo1}.

\vspace{5mm}
\noindent{\it Acknowledgements.} The first author thanks C. McMullen
for the communications on K\"ahler metrics on Teichm\"uller space and
the suggestions. We thank an anonymous referee for the suggestions on
an earlier version of this paper.

\section{K\"ahler metric and Griffiths negative vector bundle}\label{sec1}

In this section, we will review the horizontal and vertical decomposition associated with a Hermitian vector bundle, the definition of a Griffiths negative vector bundle, and the K\"ahler metric on the total space. 

\subsection{Horizontal-vertical decomposition}
Let $M$ be a complex manifold of complex dimension $n$. Let $\pi:E\to M$ be a holomorphic vector bundle of rank $r$ and the induced map $\pi_*:TE\to TM$. Let $\{e_1,\cdots,e_r\}$ be a local holomorphic frame for $E$ over a local holomorphic coordinate system $(U;z=(z^1,\cdots,z^n))$.
Each element of $E|_U$ has the form $v=\sum_{i=1}^r v^i e_i$, and so $\{\p/\p z^1,\cdots, \p/\p z^n;\p/\p v^1,\cdots, \p/\p v^r\}$ is a local holomorphic frame of $TE|_{\pi^{-1}(U)}$. The vertical subbundle of $TE$ is defined as the kernel of $\pi_*$, i.e. 
\begin{align*}
\begin{split}
  \mathcal{V}:=\mathrm{Ker}\, \pi_*\subset TE,
 \end{split}
\end{align*}
which is a holomorphic subbundle of rank $r$. In terms of local coordinates, the vertical subbundle is given by
\begin{align*}
\begin{split}
  \mathcal{V}=\mathrm{Span}_{\mb{C}}\left\{\frac{\p}{\p v^i},i=1,\cdots,r\right\}.
 \end{split}
\end{align*}
Let $G=(G_{i\b{j}})$ be a Hermitian metric on $E$. It  induces a Hermitian metric on $\mc{V}$,
\begin{align*}
\begin{split}
  \left\langle V,W\right\rangle=\sum_{i=1}^r V^i\o{W}^jG_{i\b{j}},
 \end{split}
\end{align*}
where $V=\sum_{i=1}^r V^i\p/\p v^i$, $W=\sum_{i=1}^r W^i\p/\p v^i\in\mc{V}$. Let $\n^{\mc{V}}$ denote the Chern connection of the Hermitian metric on the holomorphic vertical subbundle $\mc{V}$, and denote by
\begin{align*}
\begin{split}
  P=\sum_{i=1}^rv^i\frac{\p}{\p v^i}
 \end{split}
\end{align*}
the Euler vector field, which is a holomorphic section of $\mc{V}$. It is also called the tautological section. Then the horizontal subbundle $\mc{H}$ is defined as
\begin{align*}
\begin{split}
  \mc{H}=\{X\in TE: \n_{X}^{\mc{V}}P=0\}. 
 \end{split}
\end{align*}

We adopt now the Einstein  summation convention in the subsequent
text. We denote $G:=G_{i\b{j}}v^i\b{v}^j$, and the differentiation of
$G$
with respect to $v^{i}, \bar{v}^{j}, 
z^{\alpha}, \bar{z}^{\beta},
1 \leq i, j \leq r, \, 1 \leq \alpha, \beta 
\leq n$,  as
$$
\begin{aligned}
G_{i} &=\partial G / \partial v^{i}, \quad G_{\bar{j}}=\partial G / \partial \bar{v}^{j}, \quad G_{i \bar{j}}=\partial^{2} G / \partial v^{i} \partial \bar{v}^{j}, \\
G_{i \alpha} &=\partial^{2} G / \partial v^{i} \partial z^{\alpha}, \quad G_{i \bar{j} \bar{\beta}}=\partial^{3} G / \partial v^{i} \partial \bar{v}^{j} \partial \bar{z}^{\beta}, \quad \text { etc. },
\end{aligned}
$$

We set
\begin{align*}
\begin{split}
  \frac{\delta}{\delta z^\alpha}:=\frac{\p}{\p z^\alpha}-G_{\alpha\b{l}}G^{\b{l}i}\frac{\p}{\p v^i}.
 \end{split}
\end{align*}
Then the horizontal subbundle $\mc{H}$ can be described as
\begin{align*}
\begin{split}
  \mc{H}=\mathrm{Span}_{\mb{C}}\left\{\frac{\delta}{\delta z^\alpha},\alpha=1,\cdots,n\right\}.
 \end{split}
\end{align*}
The differential $\pi_*:\mc{H}\to TM$ is an isomorphism, and we have the following horizontal and vertical decomposition 
\begin{align*}
\begin{split}
  TE=\mc{H}\oplus \mc{V}.
 \end{split}
\end{align*}
Denote 
\begin{align}\label{deltav}
\begin{split}
  \delta v^i=dv^i+G_{\alpha\b{l}}G^{\b{l}i}dz^\alpha.
 \end{split}
\end{align}
The dual bundle $T^*E$  has now a horizontal and vertical decomposition as follows
\begin{align*}
\begin{split}
  T^*E=\mc{H}^*\oplus \mc{V}^*,
 \end{split}
\end{align*}
where 
\begin{align*}
\begin{split}
  \mc{H}^*=\mathrm{Span}_{\mb{C}}\left\{dz^\alpha,\alpha=1,\cdots,n\right\},\quad \mc{V}^*=\mathrm{Span}_{\mb{C}}\left\{\delta v^i,i=1,\cdots,r\right\}.
 \end{split}
\end{align*}

\subsection{Griffiths negative vector bundle}

Let $(E,G)$ be a holomorphic Hermitian vector bundle over a complex manifold $M$. The Chern curvature tensor $R$ is given by
\begin{align*}
\begin{split}
  R=(R_{i\b{j}\alpha\b{\beta}}G^{\b{j}k}dz^\alpha\wedge d\b{z}^\beta)e^i\otimes e_k\in A^{1,1}(M,\mathrm{End}(E)),
 \end{split}
\end{align*}
where 
\begin{align*}
	R_{i\b{j}\alpha\b{\beta}}:=-\p_\alpha\p_{\b{\beta}}G_{i\b{j}}+G^{k\b{l}}\p_\alpha G_{i\b{l}}\p_{\b{\beta}}G_{k\b{j}}.
\end{align*}
\begin{defn}\label{Griffiths}
The Hermitian vector bundle $(E,G)$ is called {\it Griffiths positive} if 
	\begin{align*}
\begin{split}
  R_{i\b{j}\alpha\b{\beta}}v^i\b{v}^j\xi^\alpha\b{\xi}^\beta>0 
 \end{split}
\end{align*}
for any nonzero $v=v^i e_i\in E$ and $\xi=\xi^\alpha\p/\p z^\alpha\in TM$. 
It is called {\it Nakano positive} if 
\begin{align*}
\begin{split}
  R_{i\b{j}\alpha\b{\beta}}u^{i\alpha}\o{u^{j\beta}}>0
 \end{split}
\end{align*}
for any nonzero $u=u^{i\alpha}e_i\otimes \p/\p z^\alpha\in E\otimes
TM$. Similarly we define the corresponding semi-positivity, negativity
and
semi-negativity.
\end{defn}
Note that the metric  $G$ defines  a smooth function on the total
space $E$,
so  $\p\b{\p}G$ is a $(1,1)$-form on $E$. Moreover, 
\begin{prop}\label{decomposition}
We have 
\begin{align*}
	\p\b{\p}G=-R_{i\b{j}\alpha\b{\beta}}v^i\b{v}^jdz^\alpha\wedge d\b{z}^\beta+G_{i\b{j}}\delta v^i\wedge \delta \b{v}^j.	
	\end{align*}
\end{prop}
\begin{proof}
	This follows by direct computations,
	\begin{align*}
		&\quad -R_{i\b{j}\alpha\b{\beta}}v^i\b{v}^jdz^\alpha\wedge d\b{z}^\beta+G_{i\b{j}}\delta v^i\wedge \delta \b{v}^j\\
		&=-(-\p_\alpha\p_{\b{\beta}}G_{i\b{j}}+G^{\b{l}k}\p_\alpha G_{i\b{l}}\p_{\b{\beta}}G_{k\b{j}})v^i\b{v}^jdz^\alpha\wedge d\b{z}^\beta\\
		&\quad+ G_{i\b{j}}(dv^i+G_{\alpha\b{l}}G^{\b{l}i}dz^\alpha)\wedge (d\b{v}^j+G_{\b{\beta}k}G^{\b{j} k}d\b{z}^\beta)\\
		&=G_{\alpha\b{\beta}}dz^\alpha\wedge d\b{z}^\beta+G_{\alpha\b{j}}dz^\alpha\wedge d\b{v}^j+G_{i\b{\beta}}dv^i\wedge d\b{z}^j+G_{i\b{j}}dv^i\wedge d\b{v}^j\\
		&=\p\b{\p}G.
	\end{align*}
\end{proof}
\begin{cor}
If $(E,G)$ is a Griffiths negative vector bundle, then $\sqrt{-1}\p\b{\p}G$ is a semi-positive $(1,1)$-form on $E$, and is strictly positive on $E^o:=E-\{0\}$.	
\end{cor}
\begin{proof}
	For any nonzero $X=w^i\p/\p v^i+\xi^\alpha\p/\p z^\alpha\in TE$, by Proposition \ref{decomposition}, one has
	\begin{align*}
\begin{split}
  (\p\b{\p}G)(X,\o{X})=-R_{i\b{j}\alpha\b{\beta}}v^i\b{v}^j\xi^\alpha\b{\xi}^\beta+G_{i\b{j}}w^i\b{w}^j\geq 0,
 \end{split}
\end{align*}
where the equality if and only if 
\begin{align*}
\begin{split}
  R_{i\b{j}\alpha\b{\beta}}v^i\b{v}^j\xi^\alpha\b{\xi}^\beta=0,\quad G_{i\b{j}}w^i\b{w}^j=0.
 \end{split}
\end{align*}
Hence $w=(w^1,\cdots,w^r)=0$ and $v=(v^1,\cdots,v^r)=0$, and $\xi=(\xi^1,\cdots, \xi^n)\neq 0$. Thus, $\sqrt{-1}\p\b{\p}G\geq 0$ on $E$, and $\sqrt{-1}\p\b{\p}G>0$
on $E^o$.
\end{proof}

\subsection{K\"ahler metric on $E$}

Now we assume that $(E,G)$ is a Griffiths negative vector bundle over a  K\"ahler manifold $(M,\omega)$, where 
$$\omega=\sqrt{-1}g_{\alpha\b{\beta}}dz^\alpha\wedge d\b{z}^\beta$$
denotes the K\"ahler form. Denote 
\begin{align}\label{new metric}
\begin{split}
  \Omega:=\pi^*\omega+\sqrt{-1}\p\b{\p}G.
 \end{split}
\end{align}
\begin{prop}\label{Kahler metric}
$\Omega$ is a K\"ahler metric on the total space $E$.	
\end{prop}
\begin{proof}
  In terms of the above local coordinates $\Omega$ is
\begin{align}\label{metric1}
\Omega=\sqrt{-1}\Omega_{\alpha\b{\beta}}dz^\alpha\wedge d\b{z}^\beta+\sqrt{-1}G_{i\b{j}}\delta v^i\wedge \delta \b{v}^j,
\end{align}
where 
\begin{align}\label{1.2}
	\Omega_{\alpha\b{\beta}}:=-R_{i\b{j}\alpha\b{\beta}}v^i\b{v}^j+g_{\alpha\b{\beta}}
\end{align}
is a positive definite matrix due  to Proposition \ref{decomposition} and $g_{\alpha\bar\beta}$ being a K\"ahler metric.	
\end{proof}

Since $G$ is a smooth function on $E$, so the differential  $\p G$ of $G$
is a globally defined one-form on $E$. 
\begin{prop}
	The norm of the one-form $\p G$ with respect to $\Omega$ is given by 
\begin{align*}
\|\p G\|^2=G	
\end{align*}
	for any metric $\omega$ on $M$. In particular, 
	\begin{align*}
	\|\p G\|^2<R	
	\end{align*}
on the disk bundle $S_R=\{(z,v)\in E| G(z,v)<R\}$.
\end{prop}
\begin{proof}
By a direct calculation and the \eqref{deltav}, one has
	\begin{align*}
\p G=G_{\alpha}dz^\alpha+G_i dv^i=G_i(dv^i+G_{\alpha\b{l}}G^{\b{l}i}dz^\alpha)=G_i\delta v^i.	
\end{align*}
Its norm square with respect to the metric $\Omega$ is 
\begin{align*}
\|\p G\|^2=G_iG_{\b{j}}G^{\b{j} i}.
\end{align*}
Since $G=G_{i\b{j}}v^i\b{v}^j$, so $G_i=G_{i\b{j}}\b{v}^j$ and
\begin{align*}
	G_iG_{\b{j}}G^{\b{j} i}=G_{i\b{l}}\b{v}^lG_{k\b{j}}v^kG^{\b{j} i}=G_{k\b{l}}v^k\b{v}^l=G,
\end{align*}
which yields that 
\begin{align*}
\|\p G\|^2=G_iG_{\b{j}}G^{\b{j} i}=G,
\end{align*}
which is independent of the metric $\omega$.
\end{proof}
\begin{defn}
A two form $\omega$ is called {\it $d$-bounded} if $\omega=d\beta$ for some (locally defined) bounded one-form $\beta$.
\end{defn}
As a result, we obtain
\begin{cor}\label{cor1}
If $\omega$ is $d$-bounded, then $\Omega$ is also $d$-bounded
on any bounded domain of $E$ with
$\Omega=d(\p G+\pi^*\beta)$ and bounded one-form
$\p G+\pi^*\beta$.
\end{cor}

\section{Curvature of the K\"ahler metric on $E$}\label{sec2}
\subsection{General formulas}
In the section, we will calculate the Chern curvature of the K\"ahler metric $\Omega$ defined in \eqref{new metric}.

 Let $\n=\n'+\b{\p}$ denote the Chern connection of $\Omega$ and  
\begin{align*}
R^{\Omega}=\n^2=\n'\circ \b{\p}+\b{\p}\circ\n'\in A^{1,1}(E, \text{End}(TE))	
\end{align*}
denote the Chern curvature of $\n$. Then 
\begin{aligns}\label{1.3}
\n'\left(\frac{\delta}{\delta z^\alpha}\right) &=	\left\langle\n'\left(\frac{\delta}{\delta z^\alpha}\right),\frac{\delta}{\delta z^\beta}\right\rangle \Omega^{\b{\beta}\gamma}\frac{\delta}{\delta z^\gamma}+\left\langle\n'\left(\frac{\delta}{\delta z^\alpha}\right),\frac{\p}{\p v^j}\right\rangle G^{\b{j}i}\frac{\p}{\p v^i}\\
&=\left(\p \Omega_{\alpha\b{\beta}}-\left\langle\frac{\delta}{\delta z^\alpha},\b{\p}\left(\frac{\delta}{\delta z^\beta}\right)\right\rangle\right)\Omega^{\b{\beta}\gamma}\frac{\delta}{\delta z^\gamma}\\
&=\p \Omega_{\alpha\b{\beta}}\Omega^{\b{\beta}\gamma}\frac{\delta}{\delta z^\gamma},
\end{aligns}
where the last equality holds since $\b{\p}\left(\frac{\delta}{\delta z^\beta}\right)$ is vertical, and 
\begin{aligns}\label{1.4}
\n'\left(\frac{\p}{\p v^i}\right)&=\left\langle\n'\left(\frac{\p}{\p v^i}\right),\frac{\delta}{\delta z^\beta}\right\rangle \Omega^{\b{\beta}\gamma}\frac{\delta}{\delta z^\gamma}+\left\langle\n'\left(\frac{\p}{\p v^i}\right),\frac{\p}{\p v^j}\right\rangle G^{\b{j}k}\frac{\p}{\p v^k}\\
&=-\left\langle\frac{\p}{\p v^i},\b{\p}\left(\frac{\delta}{\delta z^\beta}\right)\right\rangle \Omega^{\b{\beta}\gamma}\frac{\delta}{\delta z^\gamma}+\p G_{i\b{j}} G^{\b{j}k}\frac{\p}{\p v^k}\\
&=G_{i\b{j}}\p(G_{k\b{\beta}}G^{\b{j} k})\Omega^{\b{\beta}\gamma}\frac{\delta}{\delta z^\gamma}+\p G_{i\b{j}} G^{\b{j}k}\frac{\p}{\p v^k}\\
&=G_{i\b{j}}\p_\alpha(G_{k\b{\beta}}G^{\b{j} k})\Omega^{\b{\beta}\gamma}dz^\alpha\otimes\frac{\delta}{\delta z^\gamma}+\p_\alpha G_{i\b{j}} G^{\b{j}k}dz^\alpha\otimes\frac{\p}{\p v^k},
\end{aligns}
where the last equality follows from the fact $G_{i\b{j}k}=0$ since $G_{i\bar j}$ is independent of fiber $v$, and $\partial_j(G_{\bar \beta k})=\partial_j(G_{\bar \beta k \bar l}\bar v^l)=\partial_{\bar\beta}\partial_j(G_{k\bar l})\bar v^l=0$.
From (\ref{1.3}) and (\ref{1.4}), the curvature $R^{\Omega}$ is 
\begin{aligns}
R^\Omega\left(\frac{\delta}{\delta z^\alpha}\right)	&=(\n'\circ \b{\p}+\b{\p}\circ\n')\left(\frac{\delta}{\delta z^\alpha}\right)\\
&=\n'\left(-\b{\p}(G_{\alpha\b{l}}G^{\b{l}i})\frac{\p}{\p v^i}\right)+\b{\p}\left(\p \Omega_{\alpha\b{\beta}}\Omega^{\b{\beta}\gamma}\frac{\delta}{\delta z^\gamma}\right)\\
&=\left(-\p\b{\p}(G_{\alpha\b{l}}G^{\b{l}k})-\p G_{i\b{j}}G^{\b{j}k}\wedge \b{\p}(G_{\alpha\b{l}}G^{\b{l}i})+\p\Omega_{\alpha\b{\beta}}\Omega^{\b{\beta}\gamma}\wedge\b{\p}(G_{\gamma\b{l}}G^{\b{l}k})\right)\frac{\p}{\p v^k}\\
&\quad+\left(\b{\p}(\p\Omega_{\alpha\b{\beta}}\Omega^{\b{\beta}\gamma})-\p(G_{k\b{\beta}}G^{\b{j} k})G_{i\b{j}}\Omega^{\b{\beta}\gamma}\wedge \b{\p}(G_{\alpha\b{l}}G^{\b{l}i})\right)\frac{\delta}{\delta z^\gamma},
\end{aligns}
and 
\begin{aligns}\label{1.5}
R^\Omega\left(\frac{\p}{\p v^i}\right)&=\b{\p}\circ\n'	\left(\frac{\p}{\p v^i}\right)\\
&=\b{\p}
\left(G_{i\b{j}}\p(G_{k\b{\beta}}
  G^{\b{j} k})\Omega^{\b{\beta}\gamma}\frac{\delta}{\delta z^\gamma}+\p G_{i\b{j}} G^{\b{j}k}\frac{\p}{\p v^k}\right)\\
&=\b{\p}(G_{i\b{j}}\p(G_{ k\b{\beta}}G^{\b{j}k})\Omega^{\b{\beta}\gamma})\frac{\delta}{\delta z^\gamma}\\
&\quad+\left(G_{i\b{j}}\p(G_{k\b{\beta}}G^{\b{j}k})\Omega^{\b{\beta}\gamma}\wedge \b{\p}(G_{\gamma\b{l}}G^{\b{l}k})+\b{\p}(\p G_{i\b{j}} G^{\b{j}k})\right)\frac{\p}{\p v^k}.
\end{aligns}
\begin{prop}\label{prop1}
	The Chern curvature $R^{\Omega}$ satisfies 
	\begin{itemize}
	\item[(i)] $\langle R^\Omega\left(\frac{\p}{\p v^i}\right),\frac{\p}{\p v^j}\rangle=(R_{i\b{l}\alpha\b{\beta}}R_{k\b{j}\gamma\b{\sigma}}v^k\b{v}^l\Omega^{\b{\beta}\gamma}+R_{i\b{j}\alpha\b{\sigma}})dz^\alpha\wedge d\b{z}^\sigma$;
	\item[(ii)]	$\langle R^{\Omega}(\frac{\delta}{\delta z^\alpha}),\frac{\delta}{\delta z^\beta}\rangle=\b{\p}(\p\Omega_{\alpha\b{\sigma}}\Omega^{\b{\sigma}\gamma})\Omega_{\gamma\b{\beta}}-R_{p\b{l}\gamma\b{\beta}}R_{k\b{q}\alpha\b{\sigma}}v^k\b{v}^lG^{\b{q} p}dz^\gamma\wedge d\b{z}^\sigma$.
	\end{itemize}
\end{prop}
\begin{proof}
(i) We compute the inner product according to (\ref{1.5}),
\begin{align*}
	\left\langle R^\Omega\left(\frac{\p}{\p v^i}\right),\frac{\p}{\p v^j}\right\rangle &=G_{i\b{q}}\p(G_{k\b{\beta}}G^{\b{q}k})\Omega^{\b{\beta}\gamma}\wedge \b{\p}(G_{\gamma\b{l}}G^{\b{l}k})G_{k\b{j}}+\b{\p}(\p G_{i\b{q}} G^{\b{q}k})G_{k\b{j}}.
\end{align*}
Note that $G_{i\bar j} G^{k\bar j}=\delta_i^k$, hence
$$\p_\alpha(G^{k\bar q}G_{i\bar q})=0=\p_\alpha(G^{k\bar q})G_{i\bar q}+ G^{k\bar q}G_{\alpha i \bar q}.$$
Then
\begin{align*}
G_{i\b{q}}\p(G_{\b{\beta}k}G^{k\b{q}})=&\left[G_{i\b{q}}(\p_\alpha(G_{\b{\beta}k})G^{k\b{q}})+
G_{i\b{q}}(G_{\b{\beta}k}\p_\alpha(G^{k\b{q}}))\right]dz^\alpha\\
=&(G_{i\bar q}G_{\alpha \bar\beta k}G^{k\bar q}-G_{\bar \beta k}G^{k\bar q}G_{\alpha i \bar q})dz^\alpha\\
=& (G_{\alpha\bar \beta i}-G_{\bar \beta k}G^{k\bar q}G_{\alpha i \bar q} )dz^\alpha.
\end{align*}
But $R_{i\bar l \alpha\bar\beta}=-\p_\alpha\p_{\b\beta}G_{i\b l}+ G^{k\b j}\p_\alpha G_{i\b j} \p_{\b \beta}G_{k\b l}$
and $G=G(z,v)=G_{i\bar j}(z)v^i \bar v^j$, hence
$$G_i=G_{i\b j}\b v_j, G_{\alpha\b\beta i}=G_{\alpha\b \beta i \b l}\b v^l,$$ and
\begin{align*}
	R_{i\b l\alpha\b \beta}\b v^l &=-G_{\alpha\b\beta i \b l}\b v^l + G^{k\b j}G_{\alpha i \b j}G_{\b \beta k \b l}\b v^l\\
	&=-G_{\alpha\b\beta i}+ G^{k\b j}G_{\alpha i\b j}G_{\b\beta k}\\
	&=-G_{\alpha\b\beta i}+ G^{k\b q}G_{\alpha i\b q}G_{\b\beta k}.
\end{align*}
Finally, we get
\begin{align}\label{1.6}
G_{i\b{q}}\p(G_{k\b{\beta}}G^{\b{q} k})=\left(G_{\alpha\b{\beta}i}-G_{k\b{\beta}}G_{i\b{q}\alpha}G^{\b{q} k}	\right)dz^\alpha=-R_{i\b{l}\alpha\b{\beta}}\b{v}^ldz^\alpha.
\end{align}
Similar calculations give
\begin{align*}
\left\langle R^\Omega\left(\frac{\p}{\p v^i}\right),\frac{\p}{\p v^j}\right\rangle=(R_{i\b{l}\alpha\b{\beta}}R_{k\b{j}\gamma\b{\sigma}}v^k\b{v}^l\Omega^{\b{\beta}\gamma}+R_{i\b{j}\alpha\b{\sigma}})dz^\alpha\wedge d\b{z}^\sigma.	
\end{align*}
(ii) Using (\ref{1.6}) we compute
\begin{aligns}
&\quad \left\langle R^{\Omega}\left(\frac{\delta}{\delta z^\alpha}\right),\frac{\delta}{\delta z^\beta}\right\rangle\\
&=\left\langle \left(\b{\p}(\p\Omega_{\alpha\b{\beta}}\Omega^{\b{\beta}\gamma})-
    \p(G_{k\b{\beta}}G^{\b{j} k})G_{i\b{j}}\Omega^{\b{\beta}\gamma}\wedge \b{\p}(G_{\alpha\b{l}}G^{\b{l}i})\right)\frac{\delta}{\delta z^\gamma},\frac{\delta}{\delta z^\beta}\right\rangle\\
&=\b{\p}(\p\Omega_{\alpha\b{\sigma}}\Omega^{\b{\sigma}\gamma})\Omega_{\gamma\b{\beta}}-R_{p\b{l}\gamma\b{\beta}}R_{k\b{q}\alpha\b{\sigma}}v^k\b{v}^lG^{\b{q} p}dz^\gamma\wedge d\b{z}^\sigma.
\end{aligns}
\end{proof}
\begin{rem}\label{curvature0}
	From (i), when evaluated on a vertical vector,  $\langle R^\Omega\left(\frac{\p}{\p v^i}\right),\frac{\p}{\p v^j}\rangle$ vanishes, that is $\langle R^\Omega(\frac{\p}{\p v^k}, \bullet)\left(\frac{\p}{\p v^i}\right),\frac{\p}{\p v^j}\rangle=0$. 	
\end{rem}
  There exists a canonical holomorphic section of $\mc{V}$, that is 
  $$P=v^i\frac{\p}{\p v^i}\in \mc{O}_E(\mc{V})$$
   which is called the tautological section (see e.g., \cite[Section
   3]{Aikou}). Denote 
   \begin{align*}
   \Psi_{\alpha\b{\beta}}=	-R_{i\b{j}\alpha\b{\beta}}v^i\b{v}^j.
   \end{align*}
   From Proposition \ref{prop1} the $(1, 1)$-form
$   \langle R^\Omega(P),P\rangle$ is
\begin{align*}
  \langle R^\Omega(P),P\rangle
  &=\left\langle R^\Omega\left(\frac{\p}{\p v^i}\right),\frac{\p}{\p v^j}\right\rangle v^i\b{v}^j\\
	&=(R_{i\b{l}\alpha\b{\beta}}R_{k\b{j}\gamma\b{\sigma}}v^k\b{v}^l\Omega^{\b{\beta}\gamma}+R_{i\b{j}\alpha\b{\sigma}})v^i\b{v}^jdz^\alpha\wedge d\b{z}^\sigma\\
	&=\left(\Psi_{\alpha\b{\beta}}\Psi_{\gamma\b{\sigma}}\Omega^{\b{\beta}\gamma}-\Psi_{\alpha\b{\sigma}}\right)dz^\alpha\wedge d\b{z}^{\sigma}.
\end{align*}
For any point $(z, v)$
outside the zero section, i.e., in the set
$$E^o:=\{(z,v)\in E; v\neq 0\},$$
the vector  $P(z,v)\neq 0$. So $(\Psi_{\alpha\b{\beta}})$ is a positive definite matrix on  $E^o$ since $(E,G)$ is Griffiths negative. Since 
\begin{align*}
\Omega_{\alpha\b{\beta}}-\Psi_{\alpha\b{\beta}}=g_{\alpha\b{\beta}}	
\end{align*}
is positive definite,
  so 
  \begin{aligns}\label{1.9}
  	\langle \sqrt{-1}R^\Omega(P),P\rangle &=\sqrt{-1}\left(\Psi_{\alpha\b{\beta}}\Psi_{\gamma\b{\sigma}}\Omega^{\b{\beta}\gamma}-\Psi_{\alpha\b{\sigma}}\right)dz^\alpha\wedge d\b{z}^{\sigma}\\
  	&\leq  \sqrt{-1}\left(\Psi_{\alpha\b{\beta}}\Psi_{\gamma\b{\sigma}}\Psi^{\b{\beta}\gamma}-\Psi_{\alpha\b{\sigma}}\right)dz^\alpha\wedge d\b{z}^{\sigma}=0.
  \end{aligns}
Thus, we obtain
\begin{prop}\label{prop2}
$\langle \sqrt{-1}R^\Omega(P),P\rangle$
is a non-positive $(1,1)$-form on $E$.
\end{prop}
\begin{rem}
Moreover,  $\langle \sqrt{-1}R^\Omega(P),P\rangle$ is a strictly negative $(1,1)$-form on $E^o$ along the horizontal directions, that is 
$$\langle R^\Omega(\xi,\bar{\xi})(P),P\rangle<0$$
for any nonzero vector $\xi=\xi^\alpha\frac{\delta}{\delta z^\alpha}\in \mc{H}_{(z,v)}$, $(z,v)\in E^o$. In fact, from (\ref{1.9}), 
$\langle R^\Omega(\xi,\bar{\xi})(P),P\rangle=0$ if and only if 
\begin{align}\label{1.10}\left(\Psi^{\b{\beta}\gamma}-\Omega^{\b{\beta}\gamma}\right)(\Psi_{\alpha\b{\beta}}\xi^\alpha)(\Psi_{\gamma\b{\sigma}}\b{\xi}^\sigma)=0.\end{align}
Since $\left(\Psi^{\b{\beta}\gamma}-\Omega^{\b{\beta}\gamma}\right)$ is positive definite on $E^o$,  (\ref{1.10}) is equivalent to 
$$\Psi_{\alpha\b{\beta}}\xi^\alpha=0.$$
On the other hand, $(\Psi_{\alpha\b{\beta}})$ is positive definite on $E^o$ by Griffiths negativity of $(E,G)$, which implies that $\xi=0$.  
\end{rem}

The Ricci curvature of the K\"ahler metric is 
\begin{aligns}
  \mr{Ric}^\Omega&:=\text{Tr}(R^\Omega)=G^{\b{j} i}\left\langle R^\Omega\left(\frac{\p}{\p v^i}\right),\frac{\p}{\p v^j}\right\rangle+
  \Omega^{\b{\beta}\alpha
  }\left\langle R^{\Omega}\left(\frac{\delta}{\delta z^\alpha}\right),\frac{\delta}{\delta z^\beta}\right\rangle\\
&=G^{\b{j} i}(R_{i\b{l}\alpha\b{\beta}}R_{k\b{j}\gamma\b{\sigma}}v^k\b{v}^l\Omega^{\b{\beta}\gamma}+R_{i\b{j}\alpha\b{\sigma}})dz^\alpha\wedge d\b{z}^\sigma\\
&\quad+\Omega^{
  \b{\beta}
  \alpha
}
(\b{\p}(\p\Omega_{\alpha\b{\sigma}}\Omega^{\b{\sigma}\gamma})\Omega_{\gamma\b{\beta}}-R_{p\b{l}\gamma\b{\beta}}R_{k\b{q}\alpha\b{\sigma}}v^k\b{v}^lG^{
  \b{q}
  p
}
dz^\gamma\wedge d\b{z}^\sigma)\\
&=\b{\p}\p\log \det(G_{i\b{j}})+\b{\p}\p\log\det (\Omega_{\alpha\b{\beta}})\\
&=\b{\p}\p\log \left(\det(G_{i\b{j}})\cdot\det (\Omega_{\alpha\b{\beta}})\right).
\end{aligns}
 
Denote by $\iota: M\to E$ the natural embedding (as the zero section of $E$), then 
 \begin{align*}
 \iota^*(\mr{Ric}^{\Omega})=\iota^*(\b{\p}\p\log \left(\det(G_{i\b{j}})\cdot\det (\Omega_{\alpha\b{\beta}})\right))=\b{\p}\p\log \left(\det(G_{i\b{j}})\cdot\det (g_{\alpha\b{\beta}})\right)	
 \end{align*}
is the $(1,1)$-form on $M$.

\subsection{The case of Teichm\"uller space}
We specify our formulas
above for $E\to M$ to the case of 
Teichm\"uller space $M$ with $E$ its tangent and cotangent bundle,
and derive some known results; see e. g.
\cite{Wol}. 

Let $S$ be a closed surface. The holomorphic tangent bundle of  Teichm\"uller space $\cal T(S)$
is a holomorphic vector bundle  $\cal B(S)$ over Teichm\"uller space whose fiber over $X$ is the set of harmonic Beltrami differentials $B(X)$. The cotangent bundle of $\cal T(S)$ is 
$\cal Q(S)$ whose fiber over $Y\in \cal T(S)$ is $Q(Y)$, the set of holomorphic quadratic differentials over $Y$. 
\begin{defn}
 For a harmonic Beltrami differential $\mu=\mu(z)\frac{d\bar z}{dz}$ over $X$ with a hyperbolic metric $g=\rho(z)|dz|$, the $L^2$-norm is defined by 
$$||\mu||^2=\int_X |\mu(z)|^2 \rho(z)^2 |dz|^2.$$
 The $L^2$-norm of a quadratic differential $\phi=\phi(z)dz^2$ over $X$ is 
$$||\phi||^2=\int_X \frac{|\phi(z)|^2}{\rho(z)^2}|dz|^2.$$
The $L^\infty$-norm  of a quadratic differential $\phi=\phi(z)dz^2$ over $X$ is defined by 
$$||\phi||_\infty=\sup_{X} \rho^{-2}|\phi(z)|.$$
\end{defn}
 This $L^2$-norm defines the Weil-Petersson metric 
 $$\Vert \mu\Vert_{\mr{WP}}^2=\int_X |\mu(z)|^2 \rho(z)^2 |dz|^2$$ on the tangent space of $\cal T(S)$.

Now let $E=TM$
with  the Weil-Petersson metric $(G_{i\bar{j}})=(g_{\alpha\b{\beta}})$.
For any unit  vector $\xi\in E=TM$, i.e. $\|\xi\|^2=1$,
$\iota^*(\mr{Ric}^{\Omega}) $ above becomes
\begin{aligns}\label{1.7}
  \iota^*(\mr{Ric}^{\Omega})
  (\xi,\bar{\xi}) &=	\b{\p}\p\log \left(\det(g_{\alpha\bar{\beta}})\cdot\det (g_{\alpha\b{\beta}})\right)(\xi,\bar{\xi}) \\
&=2 \mr{Ric}(\xi,\bar{\xi}),
\end{aligns}
where $\mr{Ric}:=\b{\p}\p\log \det (g_{\alpha\b{\beta}})$ denotes the Ricci curvature of Weil-Petersson metric. From \cite[Lemma 4.6 (i)]{Wol}, the Ricci curvature of the Weil-Petersson metric satisfies
\begin{align}\label{1.8}
	\mr{Ric}(\xi,\bar{\xi})\leq -\frac{1}{2\pi(g-1)}.
\end{align}
where $g$ denotes the genus of Riemann surfaces. Substituting (\ref{1.8}) into (\ref{1.7}), we obtains
\begin{align*}
\iota^*(\mr{Ric}^{\Omega})(\xi,\bar{\xi})\leq -\frac{1}{\pi(g-1)}.	
\end{align*}
Thus
\begin{prop}\label{prop3}
	Let $(M,\omega)$ be Teichm\"uller space with the Weil-Petersson metric, and let $E=TM$ be the holomorphic tangent bundle. When restricts to $M$, the Ricci curvature of $\Omega$ is bounded from above by   $-\frac{1}{\pi(g-1)}$.
\end{prop}

\section{Estimates of Nakano curvature operator}\label{sec5}

In this section we shall obtain estimes of
on the Nakano curvature operator.
\subsection{Nakano curvature operator and its lowest eigenvalue}

\begin{defn}
  Let $\pi:(E,h)\to M$ be a Hermitian
  holomorphic bundle over a complex manifold
$M$ with a Hermitian metric $g$ and   
 $R$ be the curvature tensor. The Nakano curvature 
  operator $Q$ is defined as the quadratic form 
  on $TM\otimes E$, 
  $$Q(A,B)=\sum_{\alpha,j}a^{\alpha i}\o{b^{\beta j}}\left\langle R(\frac{\p}{\p z^\alpha},\frac{\p}{\p \b{z}^\beta})e_i,e_j\right\rangle$$
  for any $A=a^{\alpha i} \frac{\p}{\p z^\alpha}\otimes e_i$ and $B=b^{\beta j}\frac{\p}{\p z^\beta}\otimes e_j$ in $TM\otimes E$. 
 We identify $Q$  also with the operator on $TM\otimes E$
  $$
\langle Q(A), B\rangle_{TM\otimes E}=Q(A,B).
  $$
\end{defn}

In other words  $E$ is Nakano positive if the Hermitian form $Q$ is positive. 

\begin{rem} If the base manifold $M$ is not equipped
  with a Hermitian metric, then   the notion of Nakano
  positivity is
  well-defined  but the curvature operator $Q: T\otimes E\to T\otimes E$
  is not defined.
\end{rem}

Now we assume that $p:\mc{X}\to M$ is a holomorphic fibration with an  ample line bundle $L$ over $\mc{X}$, and 
$$E=p_*(K_{\mc{X}/M}\otimes L)$$ is the direct image bundle.
  It follows immediately from Theorem \ref{thm4} in Appendix below that
  Nakano curvature $Q$  as an operator on $TM\otimes E$
  is positive definite. Indeed let $\{\mu_\alpha\}$ be an orthonormal
  basis of $T_tM$ and let $c_0:=\min_{x\in X_t}
  \lambda_{min}(c(\phi)_{\alpha\b{\beta}})$, $X_t:=p^{-1}(t)$, $t\in M$,
  where $\lambda_{min}$  is the lowest eigenvalue of $\iota^\ast \iota$
with respect to the Hilbert space norms
in $T_tM\otimes E_t$ and $A^{n-1, 1}(X_t,L_t)$ and $\iota$ is defined by \eqref{iota}.
  Then $c_0>0$ by the ampleness of $L$. From Theorem \ref{thm4}, the Nakano curvature operator $Q$ satisfies the following
  estimate  
  $$
  Q \ge c_0 $$
  as an operator on $TM\otimes E$ with respect to a generalized Weil-Petersson metric on $TM$ and the $L^2$-metric on $E$. We shall find a more accessible and
  geometric lower
  bound for $Q$.

Let $\iota: T_tM\otimes E_t \to A^{n-1, 1}(X_t)$, $n=\dim X_t$,
be the diagonal map
\begin{align}\label{iota}
\iota:
T_tM\otimes E_t \to 
A^{0,1}(X_t,TX_t)\otimes H^{0}(X_t,L_t\otimes  K_t) 
\to A^{n-1, 1}(X_t, L_t),
\end{align}
where the first map is given by the Kodaira-Spencer tensor
and the second one is the evaluation, denoted alternatively as
$$
\iota (\mu\otimes u)(x):=(\iota_\mu u)(x),  \quad x\in X_t
$$
for any $\mu\in T_tM$ and $u\in E_t$;
 see \eqref{KS tensor}.

\begin{thm}\label{thm-estQ}
Fix $t\in M$. Suppose $L=K_{\mc{X}/M}$ is the relative canonical bundle, and 
  $L_t:=L|_{X_t}$ is a positive line bundle
 with curvature $R^{L_t}$, which gives a K\"ahler metric on $X_t$. Let
 $\sigma$ be the maximum of the eigenvalue
 of $\Box'=\n'\n'^*+\n'^*\n$ on the finite-dimensional subspace
 $\iota( T_tM\otimes E_t )$.
 Then we have
  $$
Q(A, A)\ge   (\frac 1{1+2n}  + (1+\sigma)^{-1}) 
  \Vert \iota(A)\Vert^2\geq \frac{1}{1+2n}\|\iota(A)\|^2.
  $$
  In particular, if the map $\iota$ is injective, then the Nakano curvature operator
$Q$ satisfies
  $$
Q\ge   (\frac 1{1+2n}  + (1+\sigma)^{-1})  \lambda_{min} 
$$
where
$\lambda_{min}$  is the lowest eigenvalue of $\iota^\ast \iota$
with respect to the Hilbert space norms
in $T_tM\otimes E_t$ and $A^{n-1, 1}(X_t,L_t)$.
 \end{thm}

To avoid confusion, we write
sometimes the point-wise metric square norm
of a section $u$ as
$\Vert u\Vert^2_x$, $x\in X_t$,  i.e.,
$\Vert u\Vert^2_x =|u(x)|^2 e^{-\phi(x)}$,
see the appendix Section \ref{app} for the definition of $|u(x)|^2 e^{-\phi(x)}$,
and  the $L^2$-norm as
$\Vert u\Vert $
and inner product $\langle u, v\rangle$.
We fix a point $t\in M$ and write $X=p^{-1}(t)$ for notational simplicity.

\begin{lemma}\label{lemma5.3}
 Let
 $  (L, X)$ be a positive line bundle
with $X$ being equipped with the corresponding K\"a{}hler metric.  Let $\Box=d^*d$ be the
  Laplace-Beltrami operator on
  scalar functions on $X$. We have
  \begin{enumerate}
  \item $(1+\Box)^{-1}$ preserves the positivity
    in the sense that if $f\in C^\infty(X), f\ge 0$
    then $(1+\Box)^{-1}f\ge 0$.
  \item For any   $u\in H^0(X, L)$, 
$$
\left((1+\Box)^{-1}
  \Vert u\Vert^2\right)(x) \ge \frac{1}{1+n} \Vert u\Vert_x^2,
\quad x\in X.
$$
 \end{enumerate}
\end{lemma}

\begin{proof}
    Let $f\ge 0$ be a smooth function.
  Denote $g=(1+\Box )^{-1} f $. Then
  $(1+\Box)g=f$ and $\Box g=f-g$.
  Let $x_0$ be the minimum point of $g$, $g(x_0)=\min g$. At this point $x_0$, 
   $f(x_0)-g(x_0)=\Box g(x_0)\le 0$, and
  $f(x_0)\le g(x_0)$. But $f$ is nonnegative we have
  then $0 \le g(x_0)$. This proves that $g\ge 0$ on  $X$.

Let $u\in H^0(X, L )$. We prove first 
  $$
  \Box \Vert u\Vert^2 \le k  \Vert u\Vert^2,
  $$
  where $k$ denotes the scalar curvature of $L$.
  We choose $\{e_i\}_{1\leq i\leq n}$  a local holomorphic frame of
  $TX$ and orthonormal at any fixed point, and compute
 $   \Box \|u\|^2$ at this point,
  \begin{align*}
\begin{split}
  \Box \|u\|^2&=-\sum_i\n_{\o{e_i}}\n_{e_i}\left\langle u,u\right\rangle=-\sum_i\n_{\o{e_i}}\left\langle \n_{e_i}u,u\right\rangle\\
  &=\left\langle \sum_i R(e_i,\o{e_i})u,u\right\rangle-\|\n' u\|^2\leq k\|u\|^2=n\|u\|^2,
 \end{split}
\end{align*}
where the last equality holds since the K\"ahler metric on $X$ is given by the curvature of $L$.
Hence
  $$
  (1+  \Box)
  \Vert u\Vert^2\le (1+ n)
\Vert u\Vert^2 $$
or $   \Vert u\Vert^2 \ge
\frac 1 {1+n}  (1+  \Box)  \Vert u\Vert^2
$.
Write
$v=  (1+  \Box)^{-1}\Vert u\Vert^2
$.
Then
$ (1+  \Box)v= 
\Vert u\Vert^2,
$ and the 
above becomes 
$$
(1+  \Box)v \ge \frac 1 {1+n}  (1+  \Box)  \Vert u\Vert^2,
$$
which can be rewritten as
$$
-\Box (v-\frac 1{1+n}  \Vert u\Vert^2)
\le  v-\frac 1{1+n}  \Vert u\Vert^2.
$$
The same proof above implies that
$v-\frac 1{1+n}  \Vert u\Vert^2\ge 0$. Indeed,
let $x_0$ be the minimum point of $f=v-\frac 1{1+n}  \Vert u\Vert^2$.
Thus  $-(\Box f)(x_0)\ge 0$ and thus $f(x_0)\ge 0$.
 Consequently
$v-\frac 1{1+n}  \Vert u\Vert^2\ge 0$ on $X$.
\end{proof}

We now prove Theorem \ref{thm-estQ}.
\begin{proof} 
From Theorem \ref{thm4}, the curvature of $E=p_*(L\otimes K_{\mc{X}/M})=p_*(K_{\mc{X}/M}^{\otimes 2})$  has the following form
\begin{align*}
\begin{split}
  \left\langle R(\frac{\p}{\p z^\alpha},\frac{\p}{\p \bar{z}^\beta})u,u\right\rangle=\int_{X}c(\phi)_{\alpha\b{\beta}}|u|^2e^{-\phi}+\langle(1+\Box')^{-1}i_{\b{\p}^V\frac{\delta}{\delta z^{\alpha}}}u,i_{\b{\p}^V\frac{\delta}{\delta z^{\beta}}}u\rangle.
 \end{split}
\end{align*}
Since $L|_t=K_{\mc{X}/M}|_t=K_{X_t}$ is a positive line bundle, we can choose the metric $\phi$ on $L$ such that 
\begin{align*}
\begin{split}
  e^{\phi}=\det\phi.
 \end{split}
\end{align*}
From Lemma \ref{lemma0}, one has
\begin{align*}
\begin{split}
  \left\langle R(\frac{\p}{\p z^\alpha},\frac{\p}{\p \bar{z}^\beta})u,u\right\rangle=\int_{X}(1+\Box)^{-1}(\mu_\alpha,\mu_\beta)|u|^2_{L^2}\frac{\omega^n_t}{n!}+\langle(1+\Box')^{-1}i_{\mu_\alpha}u,i_{\mu_\beta}u\rangle.
 \end{split}
\end{align*}
Here $(\cdot,\cdot)$ denotes the point-wise inner product on the holomorphic bundle $T^*X_t\otimes TX_t$, while $\left\langle \cdot,\cdot\right\rangle$ denotes the global inner product  for the sections in $A^{n-1,1}(X_t,L_t)$.

For any $A=a^{\alpha i} \frac{\p}{\p z^\alpha}\otimes e_i\in T_tM\otimes E_t$, 
then the Nakano curvature is given by
\begin{align*}
\begin{split}
  Q(A,A)&=\sum_{\alpha,j}a^{\alpha i}\o{a^{\beta j}}\left\langle R(\frac{\p}{\p z^\alpha},\frac{\p}{\p \b{z}^\beta})e_i,e_j\right\rangle\\
  &=\sum_{\alpha,j}a^{\alpha i}\o{a^{\beta j}}(\int_{X}(1+\Box)^{-1}(\mu_\alpha,\mu_\beta)(e_i,e_j)\frac{\omega^n_t}{n!}+\langle(1+\Box')^{-1}i_{\mu_\alpha}e_i,i_{\mu_\beta}e_j\rangle)\\
  &=\sum_{\alpha,j}a^{\alpha i}\o{a^{\beta j}}(\int_{X}(\mu_\alpha,\mu_\beta)(1+\Box)^{-1}(e_i,e_j)\frac{\omega^n_t}{n!}+\langle(1+\Box')^{-1}i_{\mu_\alpha}e_i,i_{\mu_\beta}e_j\rangle),
 \end{split}
\end{align*}
where the last equality holds since
\begin{align*}
\begin{split}
  \int_{X}(1+\Box)^{-1}(\mu_\alpha,\mu_\beta)(e_i,e_j)\frac{\omega^n_t}{n!}&=\int_{X}((1+\Box)^{-1}(\mu_\alpha,\mu_\beta))(1+\Box)(1+\Box)^{-1}(e_i,e_j)\frac{\omega^n_t}{n!}\\
  &=\int_{X}(1+\Box)((1+\Box)^{-1}(\mu_\alpha,\mu_\beta))(1+\Box)^{-1}(e_i,e_j)\frac{\omega^n_t}{n!}\\
  &=\int_{X}(\mu_\alpha,\mu_\beta)(1+\Box)^{-1}(e_i,e_j)\frac{\omega^n_t}{n!}.
 \end{split}
\end{align*}
From Lemma \ref{lemma5.3}, the following matrix
\begin{align*}
\begin{split}
 M_{ij}:= (1+\Box)^{-1}(e_i,e_j)-\frac{1}{1+k_1}(e_i,e_j),
 \end{split}
\end{align*}
is  semi-positive definite, where 
$$k_1= \sum_i R^{L_t^2}(e_i, \o{e_i})=2\sum_i R^{L_t}(e_i, \o{e_i})=2n.$$ On the other hand, $a^{\alpha i}\o{a^{\beta j}}(\mu_\alpha,\mu_\beta)$ is also a semi-positive matrix, so 
\begin{align*}
\begin{split}
  Q(A,A)&\geq \sum_{\alpha,\beta,i,j}a^{\alpha i}\o{a^{\beta j}}(\int_{X}\frac{1}{1+2n}(\mu_\alpha,\mu_\beta)(e_i,e_j)\frac{\omega^n_t}{n!}+\langle(1+\Box')^{-1}i_{\mu_\alpha}e_i,i_{\mu_\beta}e_j\rangle)\\
  &=\sum_{\alpha,\beta,i,j}a^{\alpha i}\o{a^{\beta j}}(\int_{X}\frac{1}{1+2n}(i_{\mu_\alpha}e_i,i_{\mu_\beta}e_j)\frac{\omega^n_t}{n!}+\langle(1+\Box')^{-1}i_{\mu_\alpha}e_i,i_{\mu_\beta}e_j\rangle)\\
  &\geq \left(\frac 1{1+2n}  + (1+\sigma)^{-1}\right) \|\sum_{\alpha,i}a^{\alpha i}i_{\mu_\alpha}e_i\|^2,
 \end{split}
\end{align*}
completing the proof. 
\end{proof}

\subsection{Nagano curvature for Teichm\"uller space}
Let $\mc{X}$ be Teichm\"uller curve over Teichm\"uller space
$M=\mc{T}$, and
$L=K_{\mc{X}/\mc{T}}$, then
$$E=p_*(K_{\mc{X}/\mc{T}}^{\otimes 2}),$$
which is exactly the cotangent bundle of $\mc{T}$. 
We fix a point $t_0$ in $ \mc{T}$ and denote $X=X_{t_0}$.
The tangent space $T_{t_0}(\mc{T})$
is the space $\mb{H}^{0,1}(X,TX)=\mb{H}^{0,1}(X,K_X^{-1})$ of harmonic Beltrami differentials, $\mu=\mu(v) \frac{\partial}{\partial v} d\bar v$.
They can be further identified with $H^0(X, K_X^2)$
of holomorphic quadratic forms $q=q(v)dv^2$
by the metric $e^{-\phi}$ on $K_X$, $$\mu=
\overline{q(v)}e^{-\phi} \frac{\partial}{\partial v} \otimes d\bar v=\frac{\b{q}}{g},\quad g=e^{\phi(v)}dv\otimes d\b{v}.$$
We shall hereafter fix this realization. In the case of Riemann
surfaces here, we have

\begin{lemma} The paring $\iota:
 \mb{H}^{0,1}(X,K_X^{-1})
\otimes H^0(X, K^{2}_X)\to A^{0,1}(X,K_X)$
  is an  isometric embedding, i.e. $\iota$ is injective and preserves the natural global inner products of $\mb{H}^{0,1}(X,K_X^{-1})
\otimes H^0(X, K^{2}_X)$ and $A^{0,1}(X,K_X)$.
  \end{lemma}  
\begin{proof}
For any 
$$A=\sum_{i,j}a^{ij}\mu_i\otimes q_j=\sum_{i,j}a^{ij}\frac{\b{q}_i}{g}\otimes q_j\in \mb{H}^{0,1}(X,K_X^{-1})
\otimes H^0(X, K^{2}_X)$$
then 
\begin{align*}
\begin{split}
  \iota(A)=\sum_{i,j}a^{ij}\frac{\b{q}_i\otimes q_j}{g}\in A^{0,1}(X,K_X)
 \end{split}
\end{align*}
which implies that $\iota$ is injective. Moreover, one has 
\begin{align*}
\begin{split}
  \|\iota(A)\|^2=\|A\|^2,
 \end{split}
\end{align*}
which completes the proof.
  \end{proof}  

From Theorem \ref{thm-estQ} and using $\dim X_t=1$, one has
\begin{cor} We have the following lower estimate
  for the Nakano curvature operator,
  $$  Q(A, A)\ge (\frac 13  + \frac 1{1+\sigma})
  \Vert A\Vert^2.$$
    As an operator, we have
  $ Q\ge \frac 13  + \frac 1{1+\sigma}$.  In particular, if $A$ has the form $A=\sum_i\lambda_i \frac{\o{q_i}}{g}\otimes q_i$ with $\lambda_i\geq 0$, then 
  $$Q(A,A)\geq \frac{2}{3}\|A\|^2.$$
\end{cor}  
\begin{proof}
We need to prove the last part. If we consider 
$$A=\sum_i \lambda_i\frac{\o{q_i}}{g}\otimes q_i$$
with $\lambda_i\geq 0$, then 
\begin{align*}
\begin{split}
  \frac{\iota(A)}{g}=\sum_i\lambda_i\frac{|q_i|^2}{g^2}=\sum_i\lambda_ie^{-2\phi}|q_i(v)|^2,
 \end{split}
\end{align*}
which follows that $\iota(A)/g$ is real.
Using the above argument as in the proof of Theorem
\ref{thm-estQ} (see e.g.  \cite[Lemma 5.1]{Wolf}), we have 
\begin{align*}
\begin{split}
  (1+\Box)^{-1}(e^{-2\phi}|q_i(v)|^2)\geq \frac{1}{3}e^{-2\phi}|q_i(v)|^2. 
 \end{split}
\end{align*}
Hence 
\begin{align*}
\begin{split}
  \int_X (1+\Box)^{-1}(\frac{\iota(A)}{g})\cdot \o{\frac{\iota(A)}{g}}\omega_0\geq 
  \frac{1}{3}\int_X\left|\frac{\iota(A)}{g}\right|^2\omega_0=\frac{1}{3}\|\iota(A)\|^2=\frac{1}{3}\|A\|^2.
 \end{split}
\end{align*}
By \eqref{app-sec}, we obtain 
\begin{align*}
\begin{split}
  \left\langle (1+\Box')^{-1}\iota(A),\iota(A)\right\rangle&=\int_X (1+\Box)^{-1}(\frac{\iota(A)}{g})\cdot \o{\frac{\iota(A)}{g}}\omega_0\geq \frac{1}{3}\|A\|^2.
 \end{split}
\end{align*}
Thus
\begin{align*}
\begin{split}
  Q(A,A)\geq \frac{2}{3}\|A\|^2.
 \end{split}
\end{align*}
\end{proof}

There have
been some recent studies
on the refined properties of the Weil-Petterson curvature
at specific points on the Teichm\"u{}ller space; see \cite{BW} and references therein.

\section{K\"ahler metric on quasi-Fuchsian space}\label{sec3}

In this section, we will recall the definitions of quasi-Fuchsian
space and complex projective structures on surfaces, and we will embed the quasi-Fuchsian space into the space of complex projective structures. Then we will define a mapping class group invariant K\"ahler metric on the quasi-Fuchsian space.

\subsection{Quasi-Fuchsian space}

Recall that the isometry group
of the hyperbolic 3-space $\mb{H}^3$ can be identified
with $\mr{PSL}(2, \mb{C})$. We use the
unit ball
in $\mb{R}^3$
as a realization of $\mathbb H^3$.
The ideal boundary is then $S^2$
and is further identified with $\mathbb{CP}^1$
such that the action of $\mr{PSL}(2,\mb{C})$  on $S^2$
is the natural
extension of its isometric action on $\mb{H}^3$.

The Teichm\"uller space
$\cal T(S)$ is realized as
the space of Fuchsian representations, i.e., 
discrete and faithful representations
$\rho:\pi_1(S)\rightarrow \mr{PSL}(2,\mb{R})$ up to conjugacy. Let $\Gamma_\rho$ be the image of $\rho$, whence $\Gamma_\rho$ acts on $S^2$ by M\"obius map preserving the equator. Then  any quasi-conformal map $f$ from $S^2$ into itself induces a quasi-conformal deformation $\rho_f$ defined by
$$\rho_f(\gamma)=     f \circ \rho(\gamma)\circ f^{-1}.$$
If furthermore  $\rho_f(\gamma)$ is an element of $\mr{PSL}(2,\mb{C})$ for any $\gamma\in \pi_1(S)$ then it defines
a representation of $\pi_1(S)$ in $\mr{PSL}(2,\mb{C})$.
Collection of such quasi-conformal deformations of Fuchsian
representations is denoted $\mr{QF}(S)$ and is identified with an open set
of a character variety
 $\chi(\pi_1(S), \mr{PSL}(2,\mb{C}))$.
Hence it has a natural induced complex structure from $\chi(\pi_1(S), \mr{PSL}(2,\mb{C}))$.

If $\phi:\pi_1(S)\rightarrow  \mr{PSL}(2,\mb{C})$ is a quasi-Fuchsian representation, then $M_\phi=\mb{H}^3/\phi(\pi_1(S))$ is a quasi-Fuchsian hyperbolic 3-manifold which is homeomorphic to $S\times \R$. Then two ideal boundaries of $M_\phi$ define a pairs of points $(X, Y)\in \cal T(S)\times \cal T(\bar S)$.
This is known as Bers' simultaneous uniformization of $\mr{QF}(S)$; see \cite{Bers}. In this case, we denote $M_\phi$ by $\mr{QF}(X,Y)$.
According to Bers' uniformization, a Fuchsian representation $\rho:\pi_1(S)\to \mr{PSL}(2,\mb{R})$ whose quotient $X=\mb{H}^2/\rho(\pi_1(S))$ is a point in $\cal T(S)$ gets identified with $(X,\bar X)$.

The mapping class group $\mr{Mod}(S)$ acts on the space of representations $\rho:\pi_1(S)\to \mr{PSL}(2,\mb{C})$ by pre-composition $\phi \rho=\rho\circ \phi_*$ where $\phi\in \mr{Mod}(S)$ and $\phi_*$ is the induced homomorphism on $\pi_1(S)$. Then $\mr{Mod}(S)$ acts on $\mr{QF}(S)=\cal T(S)\times \cal T(\bar S)$ diagonally
$$
\phi \rho=\phi (X, Y)
=(\phi  X, \phi  Y).$$

\subsection{Complex projective structure}

In this subsection, we will recall the definition of complex projective structure.

A complex projective structure on $S$ is a maximal atlas $\{(\phi_i,
U_i)| \phi_i:U_i\to S^2\}$ whose transition maps $\phi_i\circ
\phi_j^{-1}$ are restrictions of complex M\"obius maps. Then the developing map $\operatorname{dev}:\widetilde S\ra S^2$ gives rise to a holonomy representation $\rho:\pi_1(S)\ra \mr{PSL}(2,\mb{C})$.
We denote the space of marked complex projective structures on $S$ by $\cal P(S)$. Since M\"obius transformations are holomorphic, a projective structure determines a complex structure on $S$. In this way, we obtain a forgetful map $$\pi:\cal P(S)\ra \cal T(S).$$ Obviously a Fuchsian representation $\rho:\pi_1(S)\ra \mr{PSL}(2,\R)\subset \mr{PSL}(2,\mb{C})$ preserving the equator of $S^2$ gives rise to an obvious projective structure
by identifying $\mb{H}^2$ with the upper and lower hemisphere of $S^2$. This gives an embedding
$$\sigma_0:\cal T(S)\ra \cal P(S).$$

More generally, for $X\in \cal T(S)$ and {
  $Z\in \pi^{-1}(X):=P(X)$}, by conformally identifying $\widetilde
X=\mb{H}^2$, we obtain a developing map $\operatorname{dev}:\mb{H}^2\ra S^2=\mathbb{CP}^1$
for $Z$. Hence the developing map can be regarded as a meromorphic
function $f=dev$
on $\mb{H}^2$. Then the Schwarzian derivative
$$S(f)=\left[\left(\frac{f''(z)}{f'(z)}\right)'-\frac{1}{2}\left(\frac{f''(z)}{f'(z)}\right)^2\right] dz^2$$ descends to $X$ as a holomorphic quadratic differential. It is known that for any element in holomorphic quadratic differentials $Q(X)$ on $X$, one can show that there exists a complex projective structure over $X$ by solving the Schwarzian linear ODE.

In this way, $\cal P(S)$ can be identified with a holomorphic vector
bundle $\cal Q(S)$ over $\cal T(S)$ whose fiber over $X$ is
$Q(X)$. In particular this
identifies {$P(X)$} with $Q(X)$ as affine spaces \cite{Dumas}, and the choice of a base point $Z_0$ in $P(X)$ gives an isomorphism $Z\ra Z-Z_0$. Hence we will choose $Z_0=\sigma_0(X)$, and $\cal T(S)$ will be identified with zero section on $\cal Q(S)$.

\subsection{K\"a{}hler metric on  quasi-Fuchsian space}

In this subsection, we will embed the quasi-Fuchsian space into the space of complex projective structures, and we will define a mapping class group invariant K\"ahler metric on the quasi-Fuchsian space. 

\subsubsection{Embedding of quas-Fuchsian space}
Recall that given $X\in \cal T(S), Y\in \cal T(\bar S)$ 
 the Bers' uniformization determines
the quasi-Fuchsian manifold $\mr{QF}(X,Y)$. 
 Then $\mr{QF}(X, Y)$ has domain of discontinuity $\Omega_+ \cup \Omega_-$ in $S^2$ with
$\Omega_+/ \mr{QF}(X,Y)=X$, and $\Omega_-/\mr{QF}(X,Y)= Y$ where $\mr{QF}(X,Y)$ is viewed as a quasi-Fuchsian representation into $\mr{PSL}(2,\mb{C})$.

As a quotient of a domain in $\mathbb{CP}^1$ by a discrete group in $\mr{PSL}(2,\mb{C})$, the surface $\Omega_-/\mr{QF}(X,Y)$ is a marked projective surface $\Sigma_Y(X)$. Then for a fixed $Y$, we obtain a quasi-Fuchsian section, called a Bers' embedding
$$\beta_Y:\cal T(S)\ra P(Y)\subset \cal P(\bar S).$$
It is known that this map
$$\mr{QF}(X,Y) \ra \Omega_-/\mr{QF}(X,Y)$$ is a homeomorphism onto its image in $\cal P(\bar S)$; see e.g. \cite{Dumas}. Under the identification of $\cal P(\bar S)$ with $\cal Q(\bar S)$ such that $\sigma_0(\cal T(\bar S))$ is a zero section,
$$\mr{QF}(X,Y) \ra \Omega_-/\mr{QF}(X,Y) - \sigma_0(Y),$$
 this embedding includes a zero section, which is the image of $\cal T(S)$.

Then by Nehari's bound \cite{Mc}  we get 
\begin{thm} The above embedding of $\mr{QF}(X,Y)$ into $Q(Y)$ is contained in a ball of radius $\frac{3}{2}$ in 
$Q(Y)$ where the norm is the $L^\infty$-norm on quadratic differentials.
\end{thm}
\begin{cor}\label{bounded}The quasi-Fuchsian space $\mr{QF}(S)$ embeds into a neighborhood of a zero section in $\cal Q(\bar S)$
which is contained in a ball of radius $9\pi(g-1)$ in $L^2$-norm on each fiber $Q(Y)$.
\end{cor}
\begin{proof}The $L^2$-norm of a quadratic differential $\phi(z)dz^2$ is given by
$$\int_Y |\phi(z)|^2 \rho(z)^{-2}  |dz|^2\leq ||\phi||_\infty^2\cdot 2\pi(2g-2)\leq 9\pi(g-1).$$
\end{proof}

\subsubsection{Vector bundle isomorphism between quadratic differentials and Beltrami differentials}

 Note that  $B(X)$ and $Q(X)$ are
 vector bundle isomorphic by the natural identification of
 differential forms with tangent vectors via the metric,
$$ \Phi=\phi(z) dz^2 \ra \beta=\beta_\Phi=\frac{\overline{\phi(z)}}{\rho^2(z)}\frac{d\bar z}{dz}.$$
The  $L^2$-norms are by definition preserved,
$$\Vert\beta\Vert^2=||\beta||_{\mr{WP}}^2=\int_X \frac{|\phi(z)|^2}{\rho^4(z)}\rho^2(z) |dz|^2=||\Phi||^2.$$

By Corollary \ref{bounded}, we get
\begin{cor}\label{bound}
Under this isomorphism between the cotangent bundle $\cal Q(S)$ and the holomorphic tangent bundle $\cal B(S)$ of $\cal T(S)$, the quasi-Fuchsian space $\mr{QF}(S)$ embeds into a neighborhood of a zero section
in $\cal B(S)$ which is contained in a ball of radius $9\pi(g-1)$ in $L^2$-norm  on each fiber $B(X)$.
\end{cor}

Now we prove  Theorem \ref{main theorem}.
\begin{proof}
Denote $\pi: \mc{B}(S)\to \mc{T}(S)$.
Since the tangent bundle $\mc{B}(S)$ of $\mc{T}(S)$ with the Weil-Petersson metric $\omega_{\mr{WP}}$ is Griffiths negative,   then the following $(1,1)$-form
$$\Omega=\pi^*\omega_{\mr{WP}}+\sqrt{-1}\p\b{\p}G$$
defines a mapping class group invariant K\"ahler metric on $\mc{B}(S)$ by Proposition \ref{Kahler metric}. From (\ref{metric1}), one sees that $\Omega$ is an extension of the Weil-Petersson metric $\omega_{\mr{WP}}$.
From \cite[Theorem 1.5]{Mc}, the Weil-Petersson metric $\omega_{\mr{WP}}$ has a bounded primitive with respect to the Weil-Petersson metric. By Corollary \ref{cor1}, the  K\"ahler metric $\Omega$ also has a bounded primitive with respect to $\Omega$.
By Remark \ref{curvature0}, the curvature vanishes along vertical direction.  And by Propositions \ref{prop2}, \ref{prop3}, the Chern curvature $R^\Omega$ of $\Omega$ is non-positive when evaluated on the tautological section $P$, and  its Ricci curvature is bounded from above by $-\frac{1}{\pi(g-1)}$ when restricted to Teichm\"uller space.

From Corollary \ref{bound},  the quasi-Fuchsian space $\mr{QF}(S)$ embeds into a neighborhood of a zero section 
in the holomorphic tangent bundle of $\mathcal{ T}(S)$ which is contained in a ball of radius $9\pi(g-1)$ in $L^2$-norm  on each fiber $B(X)$. Hence as an open set invariant by the mapping class group, $\mr{QF}(S)$ inherits such a K\"ahler metric.
\end{proof}

\section{K\"ahler metrics on other geometric structures}\label{sec4}

 Finally, to put our results in perspective
  we remark that
  the space $\cal P(S)$ of marked complex projective 
  structures is identified with the cotangent bundle of $\mathcal 
  T(S)$ and the  natural holonomy map $\cal P(S)\to 
\chi=  \chi(\pi_1(S), \mr{PSL}(2,\mb{C}))$
  to the character variety is a local biholomorphic map
by the results of Earle-Hejhal-Hubbard  \cite{Ea, He, Hu} (see also \cite[Theorem 5.1]{Dumas}).
  Thus our constructions and 
   results  are also valid for $\cal P(S)$ and its image 
   in   $\chi$. The space $\mr{QF}(S)$ of quasi-Fuchsian representations
  is also an open subset of   $\chi$,
  $\mathcal T(S)\subset \mr{QF}(S)
  \subset   \chi$, and
  it might be interesting to understand the geometry
  of character variety   $ \chi$
  using our metric on these open subsets.

  The above remark also applies to
  the  Hitchin component for any real split simple Lie group
  $G$ of real rank
  two, namely $G=\mr{SL}(3, \mb{R}), \mr{Sp}(2, \mathbb R), G_2$.
  Indeed Labourie \cite{La} generalized the construction
in \cite{KZ} of  K\"ahler
metric
for $\mr{SL}(3, \mathbb R)$ 
to the above $G$.
In this case, the Hitchin component
is proved to be a bundle over Teichm\"u{}ller space
with fiber being a space of holomorphic differentials
of degree $3, 4, 6$, respectively. 

In general, if we consider the bundle $\cal W$ over the Teichm\"uller space whose fiber over $X$ equal to $\sum_{j\geq 2}^N H^0(X,  K_X^j)$ for some integer $N\geq 2$, where $ K_X$ is the canonical line bundle of $X$, then it is Griffiths positive, and hence our method applies to its dual space $\cal W^*$.

Hence we obtain 
\begin{cor} The curvature of the K\"ahler metric on $\cal W$  vanishes along vertical directions and is non-positive along tautological sections. Such examples include the Hitchin component for real split simple Lie groups of real rank two and the space of complex projective structures over $S$.
\end{cor}

\section{Appendix: Curvature formula of Weil-Petersson metric}\label{app}

In this Appendix we recall the curvature formula of the Weil-Petersson
metric on Teichm\"uller space using our setup and notation; see \cite{Wol, Bo1, LSYY}.

Let $p:\mc{X}\to M$ be a holomorphic fibration with compact fibers. Let $L$ be a relatively ample line bundle over $\mc{X}$ with the metric $e^{-\phi}$.  We denote by
$(z;v)=(z^1,\cdots, z^n; v^1,\cdots, v^r)$ a local admissible holomorphic coordinate system of $\mc{X}$ with $p(z;v)=z$.
Denote
\begin{align}\label{horizontal}
  \frac{\delta}{\delta z^{\alpha}}:=\frac{\p}{\p z^{\alpha}}-\phi_{\alpha\b{j}}\phi^{\b{j}i}\frac{\p}{\p v^{i}}.
\end{align}
By a routine computation, one can show that $\{\frac{\delta}{\delta z^{\alpha}}\}_{1\leq \alpha\leq n}$ spans a well-defined horizontal subbundle of $T\mc{X}$.
Let $\{dz^{\alpha};\delta v^i\}$
denote the dual frame of $\left\{\frac{\delta}{\delta z^{\alpha}}; \frac{\p}{\p v^i}\right\}$. One has
$$\delta v^i=dv^i+\phi^{i\b{j}}\phi_{\b{j}\alpha}dz^{\alpha}.$$
For any metric $\phi$ on $L$ with positive curvature on each fiber, the geodesic curvature $c(\phi)$ of $\phi$  is defined by
\begin{align}\label{cphi}
  c(\phi)=\sqrt{-1}c(\phi)_{\alpha\b{\beta}} dz^{\alpha}\wedge d\b{z}^{\beta}=\left(\phi_{\alpha\b{\beta}}-\phi_{\alpha\b{j}}\phi^{i\b{j}}\phi_{i\b{\beta}}\right)\sqrt{-1} dz^{\alpha}\wedge d\b{z}^{\beta},
\end{align}
which is clearly a horizontal real $(1,1)$-form on $\mc X$. 
\begin{lemma}\label{lemma1} The following decomposition holds,
  \begin{align}
    \sqrt{-1}\p\b{\p}\phi=c(\phi)+\sqrt{-1}\phi_{i\b{j}}\delta v^i\wedge \delta \b{v}^j.
  \end{align}
\end{lemma}

Following Berndtsson (cf. \cite{Bo, Bo1}), we define the following $L^2$-metric on
the direct image bundle $E:=p_*(K_{\mc{X}/M}\otimes L)$: for any $u\in E_{z}\equiv H^0(\mc{X}_z, (L\otimes K_{\mc{X}/M})_z)$, $z\in M$, then we define
\begin{align}\label{L2 metric}
\|u\|^2=\int_{\mc{X}_z}|u|^2e^{-\phi}. 	
\end{align}
Note that $u$ can be written locally as $u=f dv\wedge e$, where $e$ is a local holomorphic frame and locally
$$|u|^2e^{-\phi}=(\sqrt{-1})^{n^2}|f|^2 |e|^2dv\wedge d\b{v}=(\sqrt{-1})^{n^2}|f|^2 e^{-\phi}dv\wedge d\b{v},$$
where $dv:=dv^1\wedge \cdots\wedge dv^n$.

\begin{thm}[{\cite[Theorem 1.2]{Bo1}}]
\label{thm4} For any $z\in M$ and let $u\in E_{z}$, one has
\begin{align}\label{cur}
\begin{split}
\langle \sqrt{-1}\Theta^{E}u,u\rangle &=\int_{p^{-1}(z)}c(\phi)|u|^2e^{-\phi}+\langle(1+\Box')^{-1}i_{\b{\p}^V\frac{\delta}{\delta z^{\alpha}}}u,i_{\b{\p}^V\frac{\delta}{\delta z^{\beta}}}u\rangle\sqrt{-1}dz^{\alpha}\wedge d\b{z}^{\beta}\\
&=\int_{p^{-1}(z)}c(\phi)|u|^2e^{-\phi}+\langle(1+\Box')^{-1}\iota(\frac{\p}{\p z^\alpha}\otimes u),\iota(\frac{\p}{\p z^\beta}\otimes u)\rangle\sqrt{-1}dz^{\alpha}\wedge d\b{z}^{\beta}.
\end{split}
\end{align}
where $\Theta^{E}$ denotes the curvature of the Chern connection on $E$ with the $L^2$ -metric defined above, here $\Box'=\n'\n'^*+\n'^*\n$ is the Laplacian on $L|_{p^{-1}(z)}$-valued forms on $p^{-1}(z)$ defined by the $(1,0)$-part of the Chern connection on $L|_{p^{-1}(z)}$, and 
\begin{align}\label{KS tensor}
\begin{split}
  \iota(\frac{\p}{\p z^\alpha}\otimes u)=i_{\b{\p}^V\frac{\delta}{\delta z^{\alpha}}}u.
 \end{split}
\end{align}
\end{thm}

Now we will derive the curvature formula of the Weil-Petersson metric by using Berndtsson's curvature formula (see \cite[Section 4.2]{Bo1}) or \cite{LSYY}.
\begin{lemma}[{Schumacher \cite[Proposition 1]{Sch}}]\label{lemma0}
	If $e^{\phi}=\det\phi$, then 
	$$(\Box+1)c(\phi)_{\alpha\b{\beta}}=(\mu_\alpha,\mu_\beta),$$
	where $\Box:=-\phi^{i\b{j}}\frac{\p^2}{\p v^{i}\p\b{v}^{j}}$, $\mu_\alpha=\b{\p}^V\frac{\delta}{\delta z^{\alpha}}$, $(\cdot,\cdot)$ denotes the point-wise inner product.  
\end{lemma}
\begin{proof}
	By direct computation, one has 
	\begin{align}
	\phi^{i\b{j}}\frac{\p^2}{\p v^{i}\p\b{v}^{j}}c(\phi)_{\alpha\b{\beta}}=(\p\b{\p}\log\det\phi)(\frac{\delta}{\delta z^\alpha},\frac{\delta}{\delta \b{z}^\beta})-(\mu_\alpha)^{i}_{\b{j}}\o{(\mu_\beta)^{k}_{\b{l}}}\phi^{\b{j}l}\phi_{i\b{k}},	
	\end{align}
where $(\mu_{\alpha})^i_{\b{j}}=-\p_{\b{j}}(\phi_{\alpha\b{k}}\phi^{\b{k}i})$. By condition $e^{\phi}=\det\phi$, one has
\begin{align}
-\Box c(\phi)_{\alpha\b{\beta}}=c(\phi)_{\alpha\b{\beta}}-(\mu_\alpha,\mu_\beta),	
\end{align}
which completes the proof. 
\end{proof}
Now we denote by $\mc{X}$ the Teichm\"uller curve over Teichm\"uller space $M=\mc{T}$,  $L=K_{\mc{X}/M}$, then $E=p_*(K_{\mc{X}/\mc{T}}^{\otimes 2})$, which is the dual bundle of $T\mc{T}$, and the dual metric of $L^2$-metric (\ref{L2 metric}) is exactly the Weil-Petersson metric. In fact, 
\begin{align*}
	\|u\|^2&=\int_{\mc{X}_z}|u|^2e^{-\phi}=\int_{\mc{X}_z}|f|^2 e^{-\phi}\sqrt{-1}dv\wedge d\b{v}\\
	&=\int_{\mc{X}_z}|f|^2\phi_{v\b{v}}^{-2}(\phi_{v\b{v}}\sqrt{-1}dv\wedge d\b{v})=\int_{\mc{X}_z}|u|^2_{\omega}\omega_z,
\end{align*}
where $\omega=\p\b{\p}\phi$, $\omega_z=\omega|_{p^{-1}(z)}$.
\begin{lemma}
For any $\alpha\in A^{0,1}(\mc{X}_z, K_{\mc{X}_z})$, then
$$\Box(\frac{\alpha}{\omega'})=\frac{1}{\omega'}\Box'\alpha.$$	
Here $\omega'=\phi_{v\b{v}}d\b{v}\otimes dv\in  A^{0,1}(\mc{X}_z, K_{\mc{X}_z})$.
\end{lemma}
\begin{proof}
We assume that $\alpha=fd\b{v}\otimes dv$, by noting that $-\sqrt{-1}\n'^*=[\Lambda,\b{\p}]$, then
	\begin{align*}
	\frac{1}{\omega'}\Box'\alpha &=	\frac{1}{\omega'}\n'^*\n'\alpha\\
	&=\frac{1}{\omega'}\n'^*\left((\p_vf-f\p_v\log\phi_{v\b{v}})dv\wedge d\b{v}\otimes dv\right)\\
	&=-\frac{1}{\omega'}\b{\p}(\frac{1}{\phi_{v\b{v}}}(\p_vf-f\p_v\log\phi_{v\b{v}}))dv\\
	&=-\frac{1}{\phi_{v\b{v}}}\p_{\b{v}}(\frac{1}{\phi_{v\b{v}}}(\p_vf-f\p_v\log\phi_{v\b{v}}))\\
	&=-\frac{1}{\phi_{v\b{v}}}\p_{\b{v}}\p_v(\frac{f}{\phi_{v\b{v}}})=\Box(\frac{\alpha}{\omega'}).
	\end{align*}
\end{proof}
Thus 
\begin{align}\label{app-sec}
\begin{split}
\langle(1+\Box')^{-1}i_{\mu_\alpha}u,i_{\mu_\beta}u\rangle &=\langle\omega'(1+\Box)^{-1}(\omega'^{-1}i_{\mu_\alpha}u),i_{\mu_\beta}u\rangle\\
&=\int_{\mc{X}_z}(1+\Box)^{-1}(\omega'^{-1}i_{\mu_\alpha}u)\cdot (\o{\omega'^{-1}i_{\mu_\beta}u})\omega_z.
\end{split}
\end{align}
From Theorem \ref{thm4}, one has
\begin{align*}
	\langle \Theta^{E}_{\alpha\b{\beta}}u,u\rangle &=\int_{\mc{X}_z}c(\phi)_{\alpha\b{\beta}}|u|^2e^{-\phi}+\langle(1+\Box')^{-1}i_{\b{\p}^V\frac{\delta}{\delta z^{\alpha}}}u,i_{\b{\p}^V\frac{\delta}{\delta z^{\beta}}}u\rangle\\
	&=\int_{\mc{X}_z}\left((1+\Box)^{-1}(\mu_\alpha\cdot \mu_{\b{\beta}})|u|^2_{\omega}+(1+\Box)^{-1}(\omega'^{-1}i_{\mu_\alpha}u)\cdot (\o{\omega'^{-1}i_{\mu_\beta}u})\right)\omega_z.
\end{align*}
 Note that $\{\mu_\alpha\}_{1\leq \alpha\leq 3g-3}\in \mb{H}^{0,1}(\mc{X}_z,K_{\mc{X}_z}^{-1})$ are harmonic, then
\begin{lemma}\label{lemma3}
	$\{u^\alpha:=h^{\alpha\b{\beta}}i_{\o{\mu_\beta}}\omega'\}_{1\leq \alpha\leq 3g-3}$ is a basis of $E=T^*\mc{T}$. Here $$h_{\alpha\b{\beta}}=\int_{\mc{X}_z}\mu_\alpha\cdot\mu_{\b{\beta}}\omega_z=\int_{\mc{X}_z}(i_{\mu_\alpha}\omega',i_{\mu_\beta}\omega')_{\omega_z}\omega_z,$$
	and $(h^{\alpha\b{\beta}})$ is the inverse matrix of $(h_{\alpha\b{\beta}})$.
\end{lemma}
 \begin{proof}
 We need to prove $\b{\p}u^\alpha=0$ along each fiber. 	Note that 
 \begin{align*}
 u^\alpha=-h^{\alpha\b{\beta}}\p_v(\phi_{\b{\beta}v}\phi_{v\b{v}}^{-1})	\phi_{v\b{v}}dv^2.
 \end{align*}
By taking $\b{\p}$, one has 
\begin{align*}
\p_{\b{v}}\left(\p_v(\phi_{\b{\alpha}v}\phi_{v\b{v}}^{-1})\phi_{v\b{v}}\right)	&=\p_{\b{v}}\p_v(\phi_{\b{\alpha}v})-\p_{\b{v}}(\phi_{\b{\alpha}v}\phi_v)\\
&=\p_v\p_{\b{\alpha}}e^{\phi}-e^{\phi}\phi_{\b{\alpha}}\phi_v-e^{\phi}\phi_{\b{\alpha}v}=0,
\end{align*}
which completes the proof.
 \end{proof}
Note that
\begin{align*}
i_{\mu_\alpha} i_{\mu_{\b{\beta}}}\omega'=(\mu_\alpha\cdot\mu_{\b{\beta}})\omega'.	
\end{align*}
Thus
\begin{align*}
	R^{\gamma\b{\delta}}_{~~\alpha\b{\beta}}:&=\langle\Theta^E_{\alpha\b{\beta}}u^\gamma,{u}^\delta\rangle\\
	&=\int_{\mc{X}_z}\left((1+\Box)^{-1}(\mu_\alpha\cdot\mu_{\b{\beta}})(u^\gamma,u^\delta)_{\omega_t}+(1+\Box)^{-1}(\omega'^{-1}i_{\mu_\alpha}u^\gamma)\cdot (\o{\omega'^{-1}i_{\mu_\beta}u^\delta})\right)\omega_z\\
	&=h^{\gamma\b{\sigma}}h^{\b{\delta}\tau}\int_{\mc{X}_t}\left((1+\Box)^{-1}(\mu_\alpha\cdot\mu_{\b{\beta}})(\mu_\tau\cdot\mu_{\b{\sigma}})+(1+\Box)^{-1}(\mu_\alpha\cdot\mu_{\b{\sigma}})\cdot (\mu_\tau\cdot\mu_{\b{\beta}})\right)\omega_z.
\end{align*}
On the other hand, $\langle u^\alpha, u^\delta\rangle=h^{\alpha\b{\delta}}$, so the curvature of Weil-Petersson metric is 
\begin{align*}
	R_{\tau\b{\sigma}\alpha\b{\beta}}&=-R^{\gamma\b{\delta}}_{~~\alpha\b{\beta}}h_{\gamma\b{\sigma}}h_{\tau\b{\delta}}\\
	&=-\int_{\mc{X}_z}\left((1+\Box)^{-1}(\mu_\alpha\cdot\mu_{\b{\beta}})(\mu_\tau\cdot\mu_{\b{\sigma}})+(1+\Box)^{-1}(\mu_\alpha\cdot\mu_{\b{\sigma}})\cdot (\mu_\tau\cdot\mu_{\b{\beta}})\right)\omega_z.
\end{align*}

\end{document}